\documentclass[11pt,notitlepage]{amsart}
\usepackage{amsmath}     % [reqno] == equations numbered on right
\usepackage{amscd}

\advance\voffset by -0.4 truein
\advance\hoffset by -0.4 truein
\def\cocoa{{\hbox{\rm C\kern-.13em o\kern-.07em C\kern-.13em o\kern-.15em A}}}

\newcommand{\PGL}{\text{\rm PGL}}
\newcommand{\SL}{\text{\rm SL}}
\newcommand{\G}{\mathbb G}
\newcommand{\wt}{\widetilde}
\renewcommand{\:}{\colon}
\newcommand{\IP}{\text{\bf P}_{\hskip-0.1cm k}}
\newcommand{\IA}{\text{\bf A}_{k}}
\newcommand{\lra}{\longrightarrow}
\renewcommand{\O}{\mathcal O}
\renewcommand{\deg}{\mathrm{deg}\,}
\newcommand{\IIP}{\check{\text{\bf P}}_{\hskip-0.1cm k}^{\hskip-0.03cm 2}}

\newtheorem{theorem}{Theorem}
\theoremstyle{definition}
\newtheorem{subsct}[theorem]{}
\theoremstyle{plain}
 
\font\smallrm=cmr8
\font\smallsc=cmcsc10
\font\smallsl=cmsl10

\begin{document}
\author
[{\smallrm EDUARDO ESTEVES and MARINA MARCHISIO}]
{Eduardo Esteves and Marina Marchisio}
\title
[{\smallrm Invariant theory of foliations of the projective plane}]
{Invariant theory of foliations\\ of the projective plane}
\begin{abstract} We study the invariant theory of singular foliations 
of the projective plane. Our first main result is that a foliation of degree 
$m > 1$ is not stable 
only if it has singularities in dimension 1 or contains an isolated 
singular point with multiplicity 
at least $(m^2-1)/(2m+1)$. Our second main result is the 
construction of an invariant map from the space of foliations 
of degree $m$ to that of curves of degree $m^2+m-2$. We describe this 
map explicitly in case $m=2$.
\end{abstract}
\maketitle

\section{Introduction}

The study of (singular) foliations of the projective plane is an old one. 
It was central in works by Darboux \cite{D} and Poincar\'e \cite{P} in 
the XIX Century. More recently, the interest in the subject has been 
revived by Jouanolou \cite{J}. It has been an active area of study ever since.

If we want to study foliations up to projective equivalence, we enter 
the realm of Invariant Theory. Though the motivation for this study 
is natural, and Invariant Theory is a classical subject, not much 
has been done so far in 
this direction. We can mention the work by Gom\'ez-Mont and Kempf 
\cite{G-MKe}, 
who have shown that a foliation whose 
all singular points have Milnor number 1 is stable. (In fact, they showed the same 
result holds for singular foliations of higher dimension spaces as well.) 
Only recently, Alc\'antara \cite{A}, \cite{AA} has characterized 
the semi-stable foliations of degree 1 and 2, and studied 
their quotient spaces.

In these notes we propose to advance this study. Our first main result is 
Theorem \ref{thm11}, which says that a foliation of degree $m > 1$ is 
nonstable (resp.~nonsemi-stable) only if it has singularities in 
dimension 1 or contains an isolated singular point of multiplicity 
at least (resp.~greater than) $(m^2-1)/(2m+1)$.

Our second main result is Theorem \ref{thm2}, which 
yields an invariant rational map $\Phi$ from the (projective) space 
of foliations of degree $m\geq 2$ to that of plane curves of degree $m^2+m-2$. 
Using this map, we can, in principle, produce invariants of foliations 
out of invariants of plane curves. However, though the invariants of plane 
curves can all be described by the symbolic method of the XIX Century, 
generators for the algebras of invariants are known only for 
very small degrees, not larger than 8. 
Since for $m\geq 3$, the curves have degree at least 10, 
the map $\Phi$ might be 
manageable only for $m=2$, in which case we are dealing with quartics. 
In this case, we describe the map explicitly in Section 4.
 
These notes report on work partly done during a visiting professorhip 
of the first named author at the Universit\`a degli Studi di Torino. That 
author would like to thank Regione Piemonte for 
financing his position. Also, 
he would like to thank the 
Dipartimento di Matematica of the Universit\`a, specially 
Prof.~Alberto Conte, for the warm hospitality extended. 
Finally, he acknowledges 
support from CNPq, Proc.~303797/2007-0 and 473032/2008-2, and FAPERJ, 
Proc.~E-26/102.769/2008 and E-26/110.556/2010.

\section{Singular foliations}

\begin{subsct} \emph{Foliations.} 
Given a smooth algebraic variety $X$ over an algebraically closed field $k$, 
a $d$-dimensional foliation of $X$ is 
a rank-$d$ 
subbundle of the tangent bundle of $X$. Typically though, these 
subbundles do not exist. For instance, take the projective plane 
$X:=\IP^2$. A subbundle of rank 1 of the tangent bundle would 
give rise to an 
exact sequence of locally free sheaves,
$$
0 \to \O_{\IP^2}(m) \to \Omega^1_{\IP^2} \to \O_{\IP^2}(n) \to 0
$$
for certain integers $m$ and $n$, and 
this sequence would split because 
$$
H^1(\IP^2,\O_{\IP^2}(m-n))=0.
$$
Thus, 
$$
\Omega^1_{\IP^2}\cong\O_{\IP^2}(m)\oplus\O_{\IP^2}(n).
$$
Then it would follow that Euler sequence,
\begin{equation}\label{euler}
0 \to \Omega^1_{\IP^2} \to \O_{\IP^2}(-1)^{\oplus 3} \to \O_{\IP^2} \to 0,
\end{equation}
would split as well, as 
$$
H^1(\IP^2,\O_{\IP^2}(-m))=H^1(\IP^2,\O_{\IP^2}(-n))=0,
$$
giving rise to a nonzero 
global section of $\O_{\IP^2}(-1)$, an absurd.
\end{subsct}

\begin{subsct} \emph{Singular foliations.} One might ask however 
not for a 
subbundle, but for a subsheaf. 
This gives rise to a \emph{singular foliation}. In other 
words, a singular foliation is a 
subsheaf of the tangent sheaf of $X$. Its 
dimension is the generic rank of the sheaf. For instance, take the 
projective plane $X:=\IP^2$. Given a singular foliation of dimension 1, 
we may replace the subsheaf by a possibly larger reflexive subsheaf. 
Since $X$ is smooth of dimension 2, this means that the subsheaf 
is locally 
free by \cite{Ok}, Lemma 1.1.10, p.~149. So, 
a singular foliation of $\IP^2$ is a nonzero (thus injective) map
\begin{equation}\label{eta}
\eta\:\O_{\IP^2}(1-m)\lra T_{\IP^2},
\end{equation}
where $T_{\IP^2}$ is the tangent sheaf of $\IP^2$. We will deal 
only with one-dimensional singular foliations of $\IP^2$ from now on, 
and will thus drop the adjective ``singular.''  

Taking duals in the Euler sequence \eqref{euler}, we obtain the exact 
sequence
\begin{equation}\label{eulerdual}
0 \lra \O_{\IP^2}\to \O_{\IP^2}(1)^{\oplus 3} \to T_{\IP^2} \to 0.
\end{equation}
Since $\text{Ext}^1(\O_{\IP^2}(1-m),\O_{\IP^2})=0$, any map $\eta$ as in 
\eqref{eta} lifts to a map 
$$
\wt\eta\:\O_{\IP^2}(1-m)\to\O_{\IP^2}(1)^{\oplus 3},
$$
which corresponds to a choice of three homogeneous polynomials 
$F$, $G$ and 
$H$ of degree $m$. 
In other words, $\eta$ induces a homogeneous vector field on 
the three-dimensional affine space $\IA^3$:
\begin{equation}\label{D}
D:=F\frac{\partial}{\partial x}+G\frac{\partial}{\partial y}+
H\frac{\partial}{\partial z}.
\end{equation}
Here $x$, $y$ and $z$ are the 
coordinates of $\IA^3$. This vector field is 
not unique, as the lifting $\wt\eta$ of $\eta$ is not, but any other 
vector field is obtaining from the above one by summing a multiple 
of the Euler field:
$$
P\Big(x\frac{\partial}{\partial x}+y\frac{\partial}{\partial y}+
z\frac{\partial}{\partial z}\Big).
$$
At any rate, we may harmlessly say that $D$, instead of $\eta$, 
is the foliation.

Conversely, given $D$ as in \eqref{D}, one can describe 
the foliation 
$\eta$ in very concrete terms: the direction given by $\eta$ at a point 
$(x:y:z)\in\IP^2$ is that of the line passing 
through $(x:y:z)$ and $(F(x,y,z):G(x,y,z):H(x,y,z))$, whenever these two 
points are distinct. 
\end{subsct}

\begin{subsct} \emph{The space of foliations.} There are thus 
many (singular) foliations. In fact, identifying foliations that differ one 
from the other by multiplication by a nonzero constant, we obtain a 
projective space,
$$
\mathbf F_m:=\mathbf P(H^0(\IP^2,T_{\IP^2}\otimes\O_{\IP^2}(m-1))).
$$
It follows from the long exact sequence in cohomology associated to 
\eqref{eulerdual} that
\begin{align*}
\dim\mathbf F_m&=3h^0(\IP^2,\O_{\IP^2}(m))-h^0(\IP^2,\O_{\IP^2}(m-1))-1\\
&=3\binom{m+2}{2}-\binom{m+1}{2}-1\\
&=m^2+4m+2.
\end{align*}
\end{subsct}

\begin{subsct} \emph{Singular points.} 
The map $\eta$ in \eqref{eta}, though 
injective, does not give rise to a subbundle. In other words, the 
degeneracy scheme of the map is nonempty. 
The degeneracy scheme is called the 
\emph{singular locus} of the foliation, and its points the 
\emph{singular points} or \emph{singularities} of the foliation. Since 
$\eta\neq 0$, the dimension of this locus is at most 1. If the 
dimension is 1, then $\eta$ decomposes in a unique way as
$$
\O_{\IP^2}(1-m)\lra\O_{\IP^2}(1-n)\lra T_{\IP^2},
$$
where the first map is multiplication by a homogeneous polynomial 
of degree $m-n$, for a certain $n<m$, and the second is a 
foliation with finite singular locus. In this case, we say that 
$\eta$ has singularities in dimension 1.

If the dimension is zero each singularity appears with a certain length in the singular locus, called its \emph{Milnor number}. Then we can use Porteous Formula 
(see \cite{Fu}, Thm.~14.4, p.~254) to compute the 
sum $\delta$ of the Milnor numbers:
\begin{align*}\label{delta}
\delta=&\int_{\IP^2}c_2(T_{\IP^2}\otimes\O_{\IP^2}(m-1))\cap[\IP^2]\\
=&\int_{\IP^2}\Bigg[\frac{c(\O_{\IP^2}(m))^3}{c(\O_{\IP^2}(m-1))}\Bigg]_2
\cap[\IP^2]\quad
\text{(Sequence \eqref{eulerdual} and Whitney Formula)}\\
=&\int_{\IP^2}\Bigg[\frac{(1+mh)^3}{1+(m-1)h}\Bigg]_2\cap[\IP^2]\quad
\text{(where $h:=c_1(\O_{\IP^2}(1))$)}\\
=&\int_{\IP^2}[(1+3mh+3m^2h^2)(1-(m-1)h+(m-1)^2h^2)]_2\cap[\IP^2]\\
=&(m-1)^2-3m(m-1)+3m^2\\
=&m^2+m+1.
\end{align*}

Another important invariant of a singular point of the foliation is 
its \emph{multiplicity}, the maximum power of the maximal ideal of the local ring of $\mathbb P^2$ at the point containing the ideal of the singular locus of the foliation.
\end{subsct}

\begin{subsct} \emph{The degree.} 
Given a singular foliation $\eta$ as in \eqref{eta}, the 
integer $m$, clearly nonnegative, 
has a geometric interpretation. Indeed, $m$ 
is the number of tangencies of $\eta$ to a general line. 
More precisely, 
given a line $L$ on $\IP^2$, we may look at the set of points 
where $\eta$ is either 
singular or assigns a line equal to $L$. 
Given a general line, this is a finite set. (Just pick a nonsingular 
point $P$ of $\eta$, and choose $L$ transversal to the line at 
$P$ given by $\eta$.) The number of points $s$ of this set, counted 
with the appropriate weights, is given by Porteous Formula, as the 
length of the degeneracy scheme of the map of vector bundles
$$
\begin{CD}
\O_{\IP^2}(1-m)|_L\oplus T_L @>(\eta|_L,\beta)>> T_{\IP^2}|_L,
\end{CD}
$$
where $\beta$ is the natural inclusion between tangent bundles. Thus
\begin{align*}
s=&\int_{L}\big(c_1(T_{\IP^2}|_L)-c_1(\O_{\IP^2}(1-m)|_L)-c_1(T_L)\big)
\cap [L]\\
=&\int_{L}\big(3h-(1-m)h-2h\big)\quad
\text{(where $h$ is the class of a point)}\\
=&m.
\end{align*}
\end{subsct}

\begin{subsct} \emph{The dual point of view.} Let
$$
\omega:=\bigwedge^2\Omega^1_{\IP^2}\cong\O_{\IP^2}(-3).
$$
The natural product map
$$
\Omega^1_{\IP^2}\otimes\Omega^1_{\IP^2}\lra \omega
$$
gives rise to an isomorphism
$$
\Omega^1_{\IP^2}\lra T_{\IP^2}\otimes\omega.
$$
Under this isomorphism, a map $\eta$ as in \eqref{eta}, 
which corresponds to a section of $T_{\IP^2}\otimes\O_{\IP^2}(m-1)$, 
corresponds to a section 
\begin{equation}\label{tau}
\tau\in H^0(\IP^2,\Omega^1_{\IP^2}\otimes\O_{\IP^2}(m+2)).
\end{equation}
Because of \eqref{euler}, this section corresponds to three 
homogeneous polynomials $A$, $B$ and $C$ of degree $m+1$ satisfying the 
relation
\begin{equation}\label{weq}
xA+yB+zC=0.
\end{equation}
We may view the polynomials as giving a homogeneous form on $\IA^3$:
\begin{equation}\label{w}
w:=Adx+Bdy+Cdz.
\end{equation}

If $\eta$ is given by $D$ as in \eqref{D}, then 
$w$ is obtained from the determinant:
$$
\left|\begin{matrix} 
x & y & z\\
F & G & H\\
dx & dy & dz
\end{matrix}\right|
$$
In other words, $A=yH-zG$, $B=zF-xH$ and $C=xG-yF$. Of course, the assignment 
$\eta\mapsto\tau$ gives rise to a (linear) isomorphism:
$$
\mathbf P(H^0(\IP^2,T_{\IP^2}\otimes\O_{\IP^2}(m-1)))\lra
\mathbf P(H^0(\IP^2,\Omega^1_{\IP^2}\otimes\O_{\IP^2}(m+2))).
$$
We may view $\mathbf F_m$ as the space on the left-hand side or that on the 
right-hand side, at our convenience. And we may harmlessly 
say that $\tau$ or $w$ is the foliation.

Geometrically, 
for each point $(a:b:c)$ of $\IP^2$ the direction at the point 
given by $\eta$ is that of the line with equation:
$$
A(a,b,c)x+B(a,b,c)y+C(a,b,c)z=0.
$$
And the singular locus of the foliation is given by $A=B=C=0$. 

Notice that, because of \eqref{weq}, the singular locus is locally given by two equations. So the following inequality holds relating the Milnor number $\mu_P$ and the multiplicity $e_P$ of a singularity $P$ of the foliation:
$$
\mu_P \geq \frac{(e_P + 1) e_P}{2} + e_P -1 = \frac{e_P^2 + 3 e_P -2}{2}.
$$

\end{subsct}

\section{The action}

\begin{subsct} \emph{The  action.} The group of automorphisms 
of $\IP^2$, namely $\PGL(3)$, acts in a 
natural way on the space of foliations. 
The action can be described very simply in geometric terms: Let $\phi$ 
be an automorphism of $\IP^2$; given a foliation $\eta$, 
the new foliation 
$\phi\cdot\eta$ assigns to every point $P\in\IP^2$ the line 
$\phi(L)$, where $L$ is the line given by $\eta$ at $\phi^{-1}(P)$. 
Algebraically, let $g$ be a 3-by-3 matrix corresponding 
to $\phi$, 
and let $w$ as in \eqref{w} correspond to $\eta$. Then $\phi\cdot\eta$ 
corresponds to $g\cdot w$, where
$$
g\cdot w=\left[\begin{matrix}
A^{g} & B^{g} & C^{g}\\
\end{matrix}\right]g^{-1}\left[\begin{matrix}
dx\\
dy\\
dz
\end{matrix}\right].
$$
(Given any polynomial $P\in k[x,y,z]$, we denote by $P^{g}$ 
the polynomial which, viewed as a function on $\IA^3$, interpreted as the space of column vectors  of dimension 3, satisfies
$$
P^{g}(v)=P(g^{-1}v)\quad\text{for each }v\in\IA^3.)
$$
\end{subsct}

\begin{subsct} \emph{Stable points.} The action of $\PGL(3)$ produces 
the same orbits as the action by $\SL(3)$, the special linear group, 
that of 3-by-3 matrices with determinant 1, induced by the natural 
surjection $\SL(3)\to\PGL(3)$. So we will consider this induced action.

Geometric Invariant Theory tells us that there is a categorical quotient 
of a certain open subset of $\mathbf F_m$, that of 
\emph{semi-stable} points. The 
semi-stable points are those for which there is an invariant homogeneous 
polynomial on the coordinates of $\mathbf F_m$ not vanishing at the 
point. And the quotient is simply the projective scheme associated 
to the (graded) algebra of invariants. Furthermore, a smaller 
open subset of $\mathbf F_m$, consisting of \emph{stable} points, 
whose orbits in the semi-stable locus are closed, admits even a 
geometric quotient, which is thus an orbit space; see \cite{FoKiM}.

To understand the quotient, it is crucial to describe the semi-stable 
points. However, it is not easy to determine them from the 
definition. A lot more manageable than the definition 
is the Hilbert--Mumford Numerical 
Criterion, by means of one-parameter subgroups. 

It was using this criterion that 
Gom\'ez-Mont and Kempf \cite{G-MKe} have shown that a foliation whose 
all singular points have Milnor number 1 is stable, that is, 
corresponds to a stable point of $\mathbf F_m$. 
And Alc\'antara \cite{A}, \cite{AA} has characterized 
the semi-stable foliations of degrees 1 and 2.

In our case, a one-parameter subgroup is a nontrivial 
homomorphism of algebraic groups $\lambda\:\G_m\to\SL(3)$, 
where $\G_m$ is the multiplicative group of the field $k$. Every 
such homomorphism is diagonalizable: there is $g\in\SL(3)$ such that 
$$
g^{-1}\lambda(t)g=\lambda_{r_1,r_2,r_3}(t),\quad
\text{where }\lambda_{r_1,r_2,r_3}(t)=\left[\begin{matrix}
t^{r_1} & 0 & 0\\
0 & t^{r_2} & 0\\
0 & 0 & t^{r_3}
\end{matrix}\right]
$$
for each $t\in\G_m$. Since $\det\lambda(t)=1$ for every $t$, 
the $r_i$ are integers such that 
$$
r_0+r_1+r_2=0.
$$
We may also assume that $r_1\geq r_2\geq r_3$. Since $\lambda$ is 
nontrivial, $r_1>0>r_3$.

Now, the space of forms $w$ as in \eqref{w}, satisfying 
\eqref{weq}, has a basis of the form:
\begin{equation}\label{basisw}
\begin{aligned}
&w_{\alpha}^1:=x^{\alpha_1}y^{\alpha_2}z^{\alpha_3}(-ydx+xdy),\\
&w_{\beta}^2:=x^{\beta_1}y^{\beta_2}z^{\beta_3}(-zdx+xdz),\\
&w_{\gamma}^3:=y^{\gamma_2}z^{\gamma_3}(-zdy+ydz),
\end{aligned}
\end{equation}
where $\alpha:=(\alpha_1,\alpha_2,\alpha_3)$ and 
$\beta:=(\beta_1,\beta_2,\beta_3)$ 
(resp.~$\gamma:=(\gamma_2,\gamma_3)$) 
run through all triples (resp.~pairs) 
of nonnegative integers summing up to $m$. This 
basis diagonalizes the action of $\lambda_{r_1,r_2,r_3}$. More 
precisely,
\begin{align*}
\lambda_{r_1,r_2,r_3}(t)\cdot w_{\alpha}^1&=
t^{-r_1(\alpha_1+1)-r_2(\alpha_2+1)-r_3\alpha_3}w_{\alpha}^1,\\
\lambda_{r_1,r_2,r_3}(t)\cdot w_{\beta}^2&=
t^{-r_1(\beta_1+1)-r_2\beta_2-r_3(\beta_3+1)}w_{\beta}^2,\\
\lambda_{r_1,r_2,r_3}(t)\cdot w_{\gamma}^3&=
t^{-r_2(\gamma_2+1)-r_3(\gamma_3+1)}w_{\gamma}^3.\\
\end{align*}

Finally, consider a point of $\mathbf F_m$, corresponding to $w$ as in 
\eqref{w}. Then, for each $g\in\SL(3)$, 
\begin{equation}\label{gw}
g\cdot w=\sum_{\alpha}a_\alpha(g)w_{\alpha}^1+\sum_\beta 
b_\beta(g)w_{\beta}^2+\sum_\gamma c_\gamma(g)w_{\gamma}^3,
\end{equation}
for unique $a_\alpha(g)$, $b_\beta(g)$ and $c_\gamma(g)$ in $k$. 
Then the Hilbert--Mumford Numerical Criterion says that $w$ is not 
stable, that is, the corresponding point on $\mathbf F_m$ is not 
stable, 
if and only if there are $g\in\SL(3)$ and integers 
$r_1,r_2,r_3$ satisfying $r_1+r_2+r_3=0$ and $0<r_1\geq r_2\geq r_3<0$ 
such that all of the following conditions hold:
\begin{equation}\label{HM}
\begin{aligned}
r_1(\alpha_1+1)+r_2(\alpha_2+1)+r_3\alpha_3\leq&0\quad
\text{if }a_\alpha(g)\neq0,\\
r_1(\beta_1+1)+r_2\beta_2+r_3(\beta_3+1)\leq&0\quad
\text{if }b_\beta(g)\neq0,\\
r_2(\gamma_2+1)+r_3(\gamma_3+1)\leq&0\quad
\text{if }c_\gamma(g)\neq0.
\end{aligned}
\end{equation}
Furthermore, $w$ is nonsemi-stable if in addition all the inequalities 
above are strict. 
\end{subsct}

\begin{theorem}\label{thm11} A foliation of degree $m>1$ is 
nonstable (resp.~nonsemi-stable) 
only if it has singularities in dimension $1$ or contains an isolated 
singular point with multiplicity 
at least (resp.~greater than) $(m^2-1)/(2m+1)$.
\end{theorem}

\begin{proof} Let $w$ as in \eqref{w} correspond to the foliation. 
Assume first that $w$ is nonstable. Then 
there are $g\in\SL(3)$ and integers 
$r_1,r_2,r_3$ satisfying 
\begin{equation}\label{rrr}
r_1+r_2+r_3=0\quad\text{and}\quad 0<r_1\geq r_2\geq r_3<0
\end{equation} 
such that \eqref{HM} holds. Since $w$ is stable if and only 
$g\cdot w$ is, and the foliation $w$ has singularities in dimension 1 
or contains an isolated singular point with a certain multiplicity 
if and only if the same holds 
for $g\cdot w$, we may assume that $g=1$, and simplify the 
notation:
$$
a_\alpha:=a_\alpha(1),\quad b_\beta:=b_\beta(1),\quad 
c_\gamma:=c_\gamma(1).
$$

We claim that either the foliation has singularities in dimension 1 or 
\begin{equation}\label{r2r3}
r_2\leq\frac{-r_3}{m+1}.
\end{equation}
Indeed, suppose \eqref{r2r3} does not hold. Let 
$\alpha=(\alpha_1,\alpha_2,\alpha_3)$ be a triple of 
nonnegative integers with $\alpha_3=0$ and $\alpha_1+\alpha_2=m$. Then, 
since $\alpha_1,r_1-r_2\geq 0$ and $-r_3,r_2>0$,
\begin{align*}
0<(r_1-r_2)\alpha_1+r_2m-r_3&=
(r_1-r_2)\alpha_1+r_2(\alpha_1+\alpha_2)-r_3\\
&=r_1(\alpha_1+1)+r_2(\alpha_2+1)+r_3\alpha_3.
\end{align*}
Thus \eqref{HM} yields 
$a_\alpha=0$.

Also, let 
$\beta=(\beta_1,\beta_2,\beta_3)$ be a triple of 
nonnegative integers with $\beta_3=0$ and $\beta_1+\beta_2=m$. Then, 
since $\beta_1,r_1-r_2\geq 0$ and $r_2,m-1>0$,
\begin{align*}
0<(r_1-r_2)\beta_1+r_2(m-1)&=
(r_1-r_2)\beta_1+r_2(\beta_1+\beta_2)-r_2\\
&=r_1(\beta_1+1)+r_2\beta_2+r_3(\beta_3+1).
\end{align*}
Thus \eqref{HM} yields 
$b_\beta=0$.

Finally, since $r_2(m+1)+r_3>0$, we have that $c_{(m,0)}=0$. But then 
it follows from \eqref{gw} that $z|w$, and thus the singular locus of the 
foliation contains a line.

Assume now that the singular locus of the foliation is finite. 
Then \eqref{r2r3} holds, 
from which we obtain
\begin{equation}\label{r1r3}
r_1-r_3=-r_2-2r_3\geq\frac{r_3}{m+1}-2r_3=-r_3\frac{2m+1}{m+1}.
\end{equation}
Let $\alpha=(\alpha_1,\alpha_2,\alpha_3)$ be a triple of 
nonnegative integers summing up to $m$. We claim:
\begin{equation}\label{a1}
\text{If }\alpha_1>\frac{m^2-1}{2m+1}\text{ then }a_\alpha=0.
\end{equation}
Indeed, if $\alpha_1>(m^2-1)/(2m+1)$ then
\begin{align*}
r_1(\alpha_1+1)+r_2(\alpha_2+1)+r_3\alpha_3&=
(r_1-r_3)\alpha_1+(r_2-r_3)\alpha_2+r_3(m-1)\\
&\geq-r_3\frac{2m+1}{m+1}\alpha_1+r_3(m-1)\\
&>0,
\end{align*}
where for the equality above we used that $\alpha_3=m-\alpha_1-\alpha_2$ 
and $r_3=-r_1-r_2$, and for the first inequality we used \eqref{r1r3}, 
$\alpha_2\geq 0$ and $r_2\geq r_3$. Thus $a_\alpha=0$ from \eqref{HM}.

Similarly:
\begin{equation}\label{b1}
\text{If }\beta_1>\frac{m^2+m+1}{2m+1}\text{ then }b_\beta=0.
\end{equation}
Indeed, if $\beta_1>(m^2+m+1)/(2m+1)$ then
\begin{align*}
r_1(\beta_1+1)+r_2\beta_2+r_3(\beta_3+1)&=
(r_1-r_3)\beta_1+(r_2-r_3)\beta_2+r_3m-r_2\\
&\geq-r_3\frac{2m+1}{m+1}\beta_1+r_3\Big(m+\frac{1}{m+1}\Big)\\
&>0,
\end{align*}
where for the first inequality above we used \eqref{r2r3}. Thus 
$b_\beta=0$ from \eqref{HM}.

Now, using \eqref{basisw} to expand \eqref{gw}, we get 
$w=Adx+Bdy+Cdz$, where 
\begin{align*}
B=&\sum_\alpha a_\alpha x^{\alpha_1+1}y^{\alpha_2}z^{\alpha_3}-
\sum_\gamma c_\gamma y^{\gamma_2}z^{\gamma_3+1},\\
C=&\sum_\beta b_\beta x^{\beta_1+1}y^{\beta_2}z^{\beta_3}+
\sum_\gamma c_\gamma y^{\gamma_2+1}z^{\gamma_3}.
\end{align*}
Let $P:=(1:0:0)$. Since $xA+yB+zC=0$, the ideal of the singular 
locus of the foliation at $P$ is generated by $B(1,y/x,z/x)$ and 
$C(1,y/x,z/x)$. Since $\gamma_2+\gamma_3=m$, it follows that the 
multiplicity of the foliation at $P$ is $\min(m+1,\xi)$ where 
\begin{align*}
\xi:=&\min\big(\min(\alpha_2+\alpha_3\,|\,a_\alpha\neq 0),
\min(\beta_2+\beta_3\,|\,b_\beta\neq 0)\big)\\
=&\min\big(\min(m-\alpha_1\,|\,a_\alpha\neq 0),
\min(m-\beta_1\,|\,b_\beta\neq 0)\big)\\
=&m-\max\big(\max(\alpha_1\,|\,a_\alpha\neq 0),
\max(\beta_1\,|\,b_\beta\neq 0)\big).
\end{align*}
(The minimum (resp. maximum) of the empty set is 
$+\infty$ (resp.~$-\infty$) by convention.) 
Thus, it follows from \eqref{a1} and \eqref{b1} that
\begin{equation}\label{xii}
\xi\geq m-\max\Big(\frac{m^2-1}{2m+1},\frac{m^2+m+1}{2m+1}\Big)
=m-\frac{m^2+m+1}{2m+1}
=\frac{m^2-1}{2m+1}.
\end{equation}

If $w$ is nonsemi-stable then the same proof works with the following 
modifications: the inequality in \eqref{r2r3} is strict while those in 
\eqref{r1r3}, \eqref{a1}, \eqref{b1} and \eqref{xii} are not.
\end{proof}

\section{The dual discriminant curve}

\begin{theorem}\label{thm2} Given a foliation of $\IP^2$ of 
degree $m\geq 2$ whose singular locus does 
not contain a double curve, the 
lines tangent to the foliation with multiplicity at least $2$ are 
parameterized by a curve on the dual plane $\IIP$ of degree $m^2+m-2$.
\end{theorem}

\begin{proof} Let $\check x$, $\check y$ and $\check z$ 
be coordinates of $\IIP$ dual to 
$x$, $y$ and $z$. The incidence variety $I\subset \IP^2\times\IIP$ is thus given 
by
$$
\check xx+\check yy+\check zz=0.
$$
Let $D$ as in $\eqref{D}$ correspond to the foliation. Requiring the above line to be 
tangent to the foliation at $(x:y:z)$ is to impose that
$$
\check xF+\check yG+\check zH=0.
$$
Let $V\subseteq I$ be the subscheme of $I$ given by the above equation. It 
parameterizes the pairs $(P,L)$ where $L$ is a line on $\IP^2$ and $P$ is a 
point on $L$ where the foliation is singular or tangent to $L$. Let 
$\pi\: V\to\IIP$ denote the projection, and let $D\subseteq V$ be the 
degeneracy locus of the natural map $\pi^*\Omega^1_{\IIP}\to\Omega^1_V$. 
Since the singular locus of the foliation contains no double curve, 
$\pi(D)\neq\IIP$. Let $C\subseteq\IIP$ be the curve such that 
$\pi_*[D]=[C]$ as cycles. This is the 
curve parameterizing lines tangent to the foliation with 
multiplicity at least 2. 

We claim that $\deg C=m^2+m-2$. 
Indeed, let $h_1$ (resp. $h_2$) be the pullback to $\IP^2\times\IIP$ of 
the hyperplane class $h$ on $\IP^2$ (resp. $\check h$ on $\IIP$). Then 
$$
[I]=h_1+h_2\quad\text{and}\quad [V]=(h_1+h_2)(mh_1+h_2)
$$
in the Chow ring of $\IP^2\times\IIP$. Now, 
$$
[D]=c_1(\Omega^1_V)\cap[V]-c_1(\pi^*\Omega^1_{\IIP})\cap [V].
$$
It follows from the Whitney Sum Formula (\cite{Fu}, Thm.~3.2(e), p.~50) 
and the Euler exact sequence that
$$
\pi^*c_1(\Omega^1_{\IIP})\cap [\IP^2\times\IIP]=-3h_2.
$$
In addition, $\Omega^1_V$ sits in the natural exact sequence,
$$
0\lra \O_V(-1,-1)\oplus\O_V(-m,-1) \lra \Omega^1_{\IP^2\times\IIP}|_V 
\lra \Omega^1_V\lra 0.
$$
Thus, applying the Whitney Sum Formula again,
\begin{align*}
c_1(\Omega^1_V)\cap [V]=&(-3(h_1+h_2)+(h_1+h_2)+(mh_1+h_2))[V]\\
=&((m-2)h_1-h_2)[V].
\end{align*}
So,
$$
[D]=((m-2)h_1+2h_2)[V]=((m-2)h_1+2h_2)(h_1+h_2)(mh_1+h_2).
$$
Since $h_1^3=h_2^3=0$, we get
\begin{align*}
[D]=&(2m+m(m-2)+(m-2))h_1^2h_2+((m-2)+2+2m)h_1h_2^2\\
=&(m^2+m-2)h_1^2h_2+3mh_1h_2^2,
\end{align*}
and thus $\pi_*[D]=(m^2+m-2)h$.
\end{proof}

\begin{subsct} \emph{Degree $2$.} Let 
$$
\mathbf C_d:=\mathbf P(H^0(\IP^2,\O_{\IP^2}(d))),
$$
the projective space parameterizing plane curves of degree $d$. It has 
dimension $(d^2+3d)/2$. By 
Theorem \ref{thm2}, there is a rational map 
$$
\Phi\:\mathbf F_m\dashrightarrow \mathbf C_{m^2+m-2}.
$$
In case $m=2$, both the target and the source of $\Phi$ have 
the same dimension, as
$$
m^2+4m+2=\frac{(m^2+m-2)^2+3(m^2+m-2)}{2}=14.
$$
In this case, the dimensions are small enough that $\Phi$ 
can be explicitly described, using \cocoa \cite{cocoa} (assuming the 
ground field $k$ has characteristic 0).

Consider a 
point of $\mathbf F_2$ given by $w$ as in \eqref{w}. As in Section 3 we 
may write
$$
w=\sum_{\alpha}a_\alpha w_{\alpha}^1+\sum_\beta b_\beta w_{\beta}^2+
\sum_\gamma c_\gamma w_{\gamma}^3,
$$
for unique $a_\alpha$, $b_\beta$ and $c_\gamma$ in $k$, 
where $\alpha:=(\alpha_1,\alpha_2,\alpha_3)$ and 
$\beta:=(\beta_1,\beta_2,\beta_3)$ 
(resp.~$\gamma:=(\gamma_2,\gamma_3)$) 
run through all triples (resp.~pairs) 
of nonnegative integers summing up to 2, and the 
$w_{\alpha}^1$, $w_{\beta}^2$ and $w_{\gamma}^3$ are given in \eqref{basisw}.

The coefficients $a_\alpha$, $b_\beta$ and $c_\gamma$ can be seen as coordinates 
of $\mathbf F_2\cong\IP^{14}$. Then the associated quartic to $w$ 
is given by:
\begin{align*}
&\Big(c_{(2,0)}c_{(0,2)}-\frac{c_{(1,1)}^2}{4}\Big)\check{x}^4
+\Big(\frac{c_{(1,1)}b_{(0,1,1)}}{2}-c_{(0,2)}b_{(0,2,0)}-c_{(2,0)}b_{(0,0,2)}
\Big)\check{x}^3\check{y}\\
+&\Big(b_{(0,2,0)}b_{(0,0,2)}+c_{(0,2)}b_{(1,1,0)}-\frac{b_{(0,1,1)}^2}{4}
-\frac{c_{(1,1)}b_{(1,0,1)}}{2}\Big)\check{x}^2\check{y}^2\\
+&\Big(\frac{b_{(1,0,1)}b_{(0,1,1)}}{2}-c_{(0,2)}b_{(2,0,0)}-b_{(0,0,2)}b_{(1,1,0)}
\Big)\check{x}\check{y}^3
+\Big(b_{(2,0,0)}b_{(0,0,2)}-\frac{b_{(1,0,1)}^2}{4}\Big)\check{y}^4\\
+&\Big(c_{(0,2)}a_{(0,2,0)}+c_{(2,0)}a_{(0,0,2)}
-\frac{c_{(1,1)}a_{(0,1,1)}}{2}\Big)\check{x}^3\check{z}\\
+&\Big(\frac{c_{(1,1)}a_{(1,0,1)}}{2}+\frac{b_{(0,1,1)}a_{(0,1,1)}}{2}
+c_{(2,0)}b_{(1,0,1)}-\frac{c_{(1,1)}b_{(1,1,0)}}{2}
-b_{(0,0,2)}a_{(0,2,0)}\\
&\,\,\,-b_{(0,2,0)}a_{(0,0,2)}-c_{(0,2)}a_{(1,1,0)}\Big)\check{x}^2\check{y}\check{z}\\
+&\Big(\frac{b_{(1,1,0)}b_{(0,1,1)}}{2}+c_{(0,2)}a_{(2,0,0)}+b_{(1,1,0)}a_{(0,0,2)}
+b_{(0,0,2)}a_{(1,1,0)}+c_{(1,1)}b_{(2,0,0)}\\
&\,\,\,-\frac{b_{(0,1,1)}a_{(1,0,1)}}{2}-\frac{b_{(1,0,1)}a_{(0,1,1)}}{2}
-b_{(0,2,0)}b_{(1,0,1)}\Big)\check{x}\check{y}^2\check{z}\\
+&\Big(\frac{b_{(1,1,0)}b_{(1,0,1)}}{2}+\frac{b_{(1,0,1)}a_{(1,0,1)}}{2}\\
&\,\,\,-b_{(2,0,0)}b_{(0,1,1)}-b_{(0,0,2)}a_{(2,0,0)}-b_{(2,0,0)}a_{(0,0,2)}
\Big)\check{y}^3\check{z}\\
+&\Big(a_{(0,2,0)}a_{(0,0,2)}+\frac{c_{(1,1)}a_{(1,1,0)}}{2}
-\frac{a_{(0,1,1)}^2}{4}-c_{(2,0)}a_{(1,0,1)}\Big)\check{x}^2\check{z}^2\\
+&\Big(\frac{a_{(1,0,1)}a_{(0,1,1)}}{2}+b_{(1,0,1)}a_{(0,2,0)}+b_{(0,2,0)}a_{(1,0,1)}
-\frac{b_{(0,1,1)}a_{(1,1,0)}}{2}-\frac{b_{(1,1,0)}a_{(0,1,1)}}{2}\\
&\,\,\,-c_{(2,0)}b_{(2,0,0)}-c_{(1,1)}a_{(2,0,0)}-a_{(0,0,2)}a_{(1,1,0)}
\Big)\check{x}\check{y}\check{z}^2\\
+&\Big(b_{(2,0,0)}b_{(0,2,0)}+b_{(0,1,1)}a_{(2,0,0)}+a_{(2,0,0)}a_{(0,0,2)}
+b_{(2,0,0)}a_{(0,1,1)}-\frac{b_{(1,1,0)}^2}{4}-\frac{a_{(1,0,1)}^2}{4}\\
&\,\,\,-\frac{b_{(1,0,1)}a_{(1,1,0)}}{2}-\frac{b_{(1,1,0)}a_{(1,0,1)}}{2}
\Big)\check{y}^2\check{z}^2\\
+&\Big(\frac{a_{(1,1,0)}a_{(0,1,1)}}{2}+c_{(2,0)}a_{(2,0,0)}-a_{(0,2,0)}a_{(1,0,1)}
\Big)\check{x}\check{z}^3\\
+&\Big(\frac{b_{(1,1,0)}a_{(1,1,0)}}{2}
+\frac{a_{(1,1,0)}a_{(1,0,1)}}{2}\\
&\,\,\,-b_{(0,2,0)}a_{(2,0,0)}-b_{(2,0,0)}a_{(0,2,0)}
-a_{(2,0,0)}a_{(0,1,1)}\Big)\check{y}\check{z}^3\\
+&\Big(a_{(2,0,0)}a_{(0,2,0)} -\frac{a_{(1,1,0)}^2}{4}\Big)\check{z}^4=0,
\end{align*}
where $\check{x}$, $\check{y}$ and $\check{z}$ are the coordinates of $\IIP$ dual to 
$x$, $y$ and $z$.
\end{subsct}

\begin{subsct} \emph{Invariants and instability.} 
Invariants for degree-2 foliations can thus be obtained from invariants 
for plane quartics by composition. 
However, the latter invariants are not completely known. In 
\cite{Dix}, Thm.~3.2, p.~286, assuming $k$ is the field of complex 
numbers, 
Dixmier produced a homogeneous system of parameters for the 
algebra of invariants of the quartics: seven homogeneous invariants of degrees 
3, 6, 9, 12, 15, 18 and 27. More invariants should be necessary. 
According to \cite{Dix}, p.~280, the algebra of invariants can be generated by 
56 invariants, though Shioda \cite{Shi}, p.~1046, 
conjectured that 13 should be enough. 

At any rate, if the foliation is not semi-stable, neither is the corresponding 
quartic. This can be seen directly from our explicit description of the 
associated quartic, as follows. 
If $w$ is not semi-stable, there are $g\in\SL(3)$ and integers 
$r_1,r_2,r_3$ satisfying 
$$
r_1+r_2+r_3=0\quad\text{and}\quad 0<r_1\geq r_2\geq r_3<0
$$
such that \eqref{HM} holds and the inequalities are strict. 
As in the proof of Theorem~\ref{thm11}, 
assume $g=1$. Then, reasoning as in the proof of that theorem, 
we can show that
\begin{equation}\label{coeff0}
a_{(2,0,0)}=a_{(1,1,0)}=a_{(1,0,1)}=b_{(2,0,0)}=b_{(1,1,0)}=0.
\end{equation}
Furthermore, either $b_{(1,0,1)}=0$ or
\begin{equation}\label{coeff00}
a_{(0,2,0)}=a_{(0,1,1)}=b_{(0,2,0)}=0.
\end{equation}
Thus, using \eqref{coeff0} to simplify the equation of the quartic, we 
get:
\begin{equation}\label{quarticu}
\begin{aligned}
&\Big(c_{(2,0)}c_{(0,2)}-\frac{c_{(1,1)}^2}{4}\Big)\check{x}^4
+\Big(\frac{c_{(1,1)}b_{(0,1,1)}}{2}-c_{(0,2)}b_{(0,2,0)}-c_{(2,0)}b_{(0,0,2)}=0
\Big)\check{x}^3\check{y}\\
+&\Big(b_{(0,2,0)}b_{(0,0,2)}-\frac{b_{(0,1,1)}^2}{4}-\frac{c_{(1,1)}b_{(1,0,1)}}{2}
\Big)\check{x}^2\check{y}^2
+\Big(\frac{b_{(1,0,1)}b_{(0,1,1)}}{2}\Big)\check{x}\check{y}^3\\
-&\Big(\frac{b_{(1,0,1)}^2}{4}\Big)\check{y}^4+
\Big(c_{(0,2)}a_{(0,2,0)}+c_{(2,0)}a_{(0,0,2)}
-\frac{c_{(1,1)}a_{(0,1,1)}}{2}\Big)\check{x}^3\check{z}\\
+&\Big(\frac{b_{(0,1,1)}a_{(0,1,1)}}{2}+c_{(2,0)}b_{(1,0,1)}-b_{(0,0,2)}a_{(0,2,0)}
-b_{(0,2,0)}a_{(0,0,2)}\Big)\check{x}^2\check{y}\check{z}\\
-&\Big(b_{(0,2,0)}b_{(1,0,1)}+
\frac{b_{(1,0,1)}a_{(0,1,1)}}{2}\Big)\check{x}\check{y}^2\check{z}
+\Big(a_{(0,2,0)}a_{(0,0,2)}-\frac{a_{(0,1,1)}^2}{4}\Big)\check{x}^2\check{z}^2\\
+&\Big(b_{(1,0,1)}a_{(0,2,0)}\Big)\check{x}\check{y}\check{z}^2=0.
\end{aligned}
\end{equation}
Then $(0:0:1)$ is a singular point of the quartic. Furthermore, 
if $b_{(1,0,1)}\neq 0$, then \eqref{coeff00} holds, and the 
equation becomes:
\begin{align*}
&\Big(c_{(2,0)}c_{(0,2)}-\frac{c_{(1,1)}^2}{4}\Big)\check{x}^4
+\Big(\frac{c_{(1,1)}b_{(0,1,1)}}{2}-c_{(2,0)}b_{(0,0,2)}
\Big)\check{x}^3\check{y}\\
-&\Big(\frac{b_{(0,1,1)}^2}{4}+\frac{c_{(1,1)}b_{(1,0,1)}}{2}
\Big)\check{x}^2\check{y}^2
+\Big(\frac{b_{(1,0,1)}b_{(0,1,1)}}{2}\Big)\check{x}\check{y}^3
-\Big(\frac{b_{(1,0,1)}^2}{4}\Big)\check{y}^4\\
+&\Big(c_{(2,0)}a_{(0,0,2)}\Big)\check{x}^3\check{z}
+\Big(c_{(2,0)}b_{(1,0,1)}\Big)\check{x}^2\check{y}\check{z}=0.
\end{align*}
In this case, the quartic has a triple point at $(0:0:1)$ with two equal 
tangent lines, or a quadruple point, whence is 
not semi-stable, according to \cite{FoKiM}, p.~80. On the other hand, 
if $b_{(1,0,1)}=0$, then the equation of the quartic becomes
\begin{align*}
&\Big(c_{(2,0)}c_{(0,2)}-\frac{c_{(1,1)}^2}{4}\Big)\check{x}^4
+\Big(\frac{c_{(1,1)}b_{(0,1,1)}}{2}-c_{(0,2)}b_{(0,2,0)}-c_{(2,0)}b_{(0,0,2)}
\Big)\check{x}^3\check{y}\\
+&\Big(b_{(0,2,0)}b_{(0,0,2)}-\frac{b_{(0,1,1)}^2}{4}
\Big)\check{x}^2\check{y}^2
+\Big(c_{(0,2)}a_{(0,2,0)}+c_{(2,0)}a_{(0,0,2)}
-\frac{c_{(1,1)}a_{(0,1,1)}}{2}\Big)\check{x}^3\check{z}\\
+&\Big(\frac{b_{(0,1,1)}a_{(0,1,1)}}{2}-b_{(0,0,2)}a_{(0,2,0)}
-b_{(0,2,0)}a_{(0,0,2)}\Big)\check{x}^2\check{y}\check{z}\\
+&\Big(a_{(0,2,0)}a_{(0,0,2)}-\frac{a_{(0,1,1)}^2}{4}\Big)\check{x}^2\check{z}^2=0,
\end{align*}
whence the union of a double line, $\check{x}=0$, and a 
conic, thus again not semi-stable, according to loc.~cit..

However, there are nonsemi-stable quartics with milder singularities 
that do not correspond to nonsemi-stable 
foliations. For instance, if we set \eqref{coeff0}, we end up with 
 Equation \eqref{quarticu} for the quartic. If we further set 
$a_{(0,2,0)}=a_{(0,1,1)}=0$, we get
\begin{equation}
\begin{aligned}
&\Big(c_{(2,0)}c_{(0,2)}-\frac{c_{(1,1)}^2}{4}\Big)\check{x}^4
+\Big(\frac{c_{(1,1)}b_{(0,1,1)}}{2}-c_{(0,2)}b_{(0,2,0)}-c_{(2,0)}b_{(0,0,2)}
\Big)\check{x}^3\check{y}\\
+&\Big(b_{(0,2,0)}b_{(0,0,2)}-\frac{b_{(0,1,1)}^2}{4}-\frac{c_{(1,1)}b_{(1,0,1)}}{2}
\Big)\check{x}^2\check{y}^2
+\Big(\frac{b_{(1,0,1)}b_{(0,1,1)}}{2}\Big)\check{x}\check{y}^3\\
-&\Big(\frac{b_{(1,0,1)}^2}{4}\Big)\check{y}^4+
\Big(c_{(2,0)}a_{(0,0,2)}\Big)\check{x}^3\check{z}
+\Big(c_{(2,0)}b_{(1,0,1)}
-b_{(0,2,0)}a_{(0,0,2)}\Big)\check{x}^2\check{y}\check{z}\\
-&\Big(b_{(0,2,0)}b_{(1,0,1)}\Big)\check{x}\check{y}^2\check{z}=0.
\end{aligned}
\end{equation}
This quartic has a triple  or quadruple point, and is thus nonsemi-stable. If we choose 
the remaining coordinates of $\mathbf F_2$ such that 
$$
b_{(1,0,1)}b_{(0,2,0)}\neq 0\quad\text{and}\quad
b_{(1,0,1)}c_{(2,0)}+a_{(0,0,2)}b_{(0,2,0)}\neq 0,
$$
then the triple point has distinct tangent lines. If we now let none of these 
lines be contained in the quartic, which is an open condition on the 
parameters that can be satisfied, as it can be easily verified with 
\cocoa \cite{cocoa}, 
then $(0:0:1)$ is the unique singular 
point of the quartic. Thus, the quartic  arises from  a semi-stable foliation. 
The above simple example shows that there are invariants of degree-2 foliations that do not arise from invariants of quartics.
\end{subsct}

\vspace{0.5cm}

{\smallsc Instituto de Matem\'atica Pura e Aplicada, 
Estrada Dona Castorina 110, 22460-320 Rio de Janeiro RJ, Brazil}

{\smallsl E-mail address: \small\verb?esteves@impa.br?}

\vspace{0.5cm}

{\smallsc Universit\`a degli Studi di Torino, Via Carlo Alberto 10, 
10123 Torino, Italy}

{\smallsl E-mail address: \small\verb?marina.marchisio@unito.it?}


\begin{thebibliography}{99}

\bibitem{A}
C. Alc\'antara,
\emph{The good quotient of the semi-stable foliations of 
$\mathbb C\mathbb P^2$ of degree $1$},
Results Math. {\bfseries 53} (2009), 1--7.

\bibitem{AA} C. Alc\'antara,
\emph{Geometric invariant theory for holomorphic foliations on  $\mathbb C\mathbb P^2$ of degree $2$},
Glasgow Math. J. {\bfseries 53} (2011), 153--168.


\bibitem{cocoa} CoCoATeam, 
{\it \cocoa: a system for doing Computations in 
Commutative Algebra}.
Available at http://cocoa.dima.unige.it.

\bibitem{D} G. Darboux,
\emph{M\'emoire sur les \'equations diff\'erentielles alg\'ebriques du premier 
ordre e du premier degr\'e (M\'elanges)},
Bull. Sci. Math\`ematiques 2\`eme s\'erie {\bfseries 2} (1878), 60--96; 
123--144; 151--200.

\bibitem{Dix} J. Dixmier,
\emph{On the projective invariants of quartic plane curves},
Adv. in Math. {\bfseries 64} (1987), 279--304.

\bibitem{FoKiM}
J. Fogarty, F. Kirwan and D. Mumford,
\emph{Geometric Invariant Theory},
Ergebnisse der Mathematik und ihrer Grenzgebiete (2), vol. 34, 
Springer-Verlag, Berlin, 1994.

\bibitem{Fu} W. Fulton,
\emph{Intersection Theory},
Ergebnisse der Mathematik und ihrer Grenzgebiete (3), vol.~2, 
Springer-Verlag, Berlin, 1984.

\bibitem{G-MKe} X. Gom\'ez-Mont and G. Kempf,
\emph{Stability of meromorphic vector fields in projective spaces},
Comment.~Math.~Helv. {\bfseries 64} (1989), 462--473.

\bibitem{J}
J.-P. Jouanolou,
\emph{\'Equations de Pfaff alg\'ebriques},
Lecture Notes in Mathematics, vol. 708, Springer-Verlag, Berlin, 1979.

\bibitem{Ok} 
C. Okonek, M. Schneider and H. Spindler,
\emph{Vector bundles on complex projective spaces}, 
Progress in Mathematics, vol. 3, Birkh\"auser, Boston, 1980.
 
\bibitem{P}
H. Poincar\'e,
\emph{Sur l'int\'egration alg\'ebrique des \'equations diff\'erentielles du 
premier ordre et du premier degr\'e}, I and II, 
R. Circ. Mat. Palermo {\bfseries 5} (1891), 161--191; {\bfseries 11} (1897), 
193--239.

\bibitem{Shi}
T. Shioda,
\emph{On the graded ring of invariants of binary octavics},
Amer. J. Math. {\bfseries 89} (1967), 1022--1046.

\end{thebibliography}
\end{document}